\documentclass[12pt]{article}
 \usepackage{amsfonts,amssymb,amsthm,amsmath,cite}

	\newcommand{\R}{\mathbb R}

	\newcommand{\om}{\operatorname{\rm Z}_1\!}  
	\newcommand{\omd}{\operatorname{{\rm Z}}^*_1} 
	\newcommand{\omp}{\operatorname{\rm Z}_p} 
	\newcommand{\ompd}{\operatorname{{\rm Z}}^*_p} 

	\newcommand{\diam}{\operatorname{diam}} 

	\newcommand{\sn}{{S^{n-1}}}  
	\newcommand{\affn}{\operatorname{Aff}(n,1)} 
	\newcommand{\rn}{{\R^{n}}}

	\newcommand{\eqnref}[1]{(\ref{#1})}

	\newtheorem{theorem}{Theorem}
	\newtheorem{lemma}{Lemma}
	
	\newtheorem{proposition}{Proposition}
	\providecommand{\norm}[1]{\lVert#1\rVert}
	\providecommand{\abs}[1]{\lvert#1\rvert}
	\providecommand{\Vol}[1]{\lvert#1\rvert}

\title{Anisotropic Fractional Sobolev Norms}
\author{Monika Ludwig}
\date{}

\begin{document}
\maketitle

\begin{abstract}

Bourgain, Brezis \& Mironescu showed that (with suitable scaling) the fractional Sobolev $s$-seminorm of a function $f\in W^{1,p}(\rn)$ converges to the Sobolev seminorm of  $f$ as $s\to 1^-$. The anisotropic $s$-seminorms of $f$ defined by a norm on $\rn$ with unit ball $K$ are shown to converge to the anisotropic Sobolev seminorm of $f$ defined by the norm with unit ball $\,\ompd K$, the polar $L_p$ moment body of $K$. The limiting behavior for $s\to 0^+$ is also determined (extending results by Maz$'$ya \& Shaposhnikova). 
\end{abstract}


\bigskip
For $p\ge 1$ and $0<s<1$, Gagliardo 
introduced the fractional Sobolev $s$-seminorm  of a function $f\in L^p(\Omega)$ as
\begin{equation}\label{gagl}
\norm{f}^p_{W^{s,p}(\Omega)} =
\int_{\Omega} \int_{\Omega} \frac{\abs{f(x)-f(y)}^p}{\abs{x-y}^{n+ps}}\,dx\,dy,
\end{equation}
where $\abs{\,\cdot\,}$ denotes the Euclidean norm on $\rn$ and $\Omega\subset \rn$.
This seminorm turned out to be critical in the study of traces of Sobolev functions in the Sobolev space $W^{1,p}(\Omega)$ (cf.\ \cite{Gagliardo}).  Fractional Sobolev norms have found numerous applications within mathematics and applied mathematics (cf.\ \cite{BourgainBrezisMironescu02,hitchhiker, Mazya}).

The limiting behavior of fractional Sobolov $s$-seminorms as $s\to 1^-$ and $s\to 0^+$ turns out to be very interesting.
Bourgain, Brezis \& Mironescu \cite{BourgainBrezisMironescu} showed that 
\begin{equation}\label{bbm}
\lim_{s\to1^-} (1-s)\norm{f}^p_{W^{s,p}(\Omega)} = \alpha_{n,p} \,\norm{ f}^p_{W^{1,p}(\Omega)}
\end{equation}
for $f\in W^{1,p}(\Omega)$ and $\Omega\subset \rn$ a smooth and bounded domain, where  $\alpha_{n,p}$  is a constant depending on $n$ and $p$, and
\begin{equation}\label{sobnorm} 
\norm{f}_{W^{1,p}(\Omega)} = \big(\int_{\Omega} \abs{\nabla f(x)}^p dx\big)^{1/p}
\end{equation}
is the Sobolev seminorm of $f$.

Maz$'$ya \& Shaposhnikova \cite{MazyaShaposhnikova} showed that if $f\in W^{s,p}(\rn)$ for all $s\in(0,1)$, where $W^{s,p}(\rn)$ are the functions in $L^p(\rn)$ with 
finite  Gagliardo seminorm \eqnref{gagl} with $\Omega=\rn$, then
\begin{equation}\label{mazyash}
\lim_{s\to 0+}s\,\norm{ f}^p_{W^{s,p}(\rn)}= \frac{2\,n}{p}\, \Vol{B}\,\abs{f}^p_p,
\end{equation}
where  $B\subset \rn$ is  $n$-dimensional Euclidean unit ball, $\Vol{B}$ its $n$-dimensional volume and $\abs{f}_p$ the $L^p$ norm of $f$ on $\rn$.  

\goodbreak

An anisotropic Sobolev seminorm is obtained by replacing the Euclidean norm $\abs{\,\cdot\,}$ in \eqnref{sobnorm} by an arbitrary norm $\norm{\,\cdot\,}_L$ with unit ball $L$. We set
\begin{equation*}\label{asobnorm} 
\norm{ f}_{W^{1,p}, K} = \big(\int_{\rn} \norm{\nabla f(x)}_{K^*}^p\, dx\big)^{1/p},
\end{equation*}
where $K^*=\{v\in\rn:v\cdot x\le1 \hbox{ for all } x\in K\}$ is the polar body of $K$. 
Anisotropic Sobolev seminorms have attracted increased interest in recent years (cf.\ \cite{AlvinoFeroneTrombettiLions,Cordero:Nazaret:Villani, Gromov_Sobolev, FigalliMaggiPratelli}).

A natural question is to study the limiting behavior of anisotropic $s$-seminorms as $s\to 1^-$ and $s\to0^+$.  While one might suspect that the limit as $s\to 1^-$ of  the anisotropic $s$-seminorms defined using a norm with unit ball $K$ is the Sobolev seminorm with the same unit ball, this turns out not to be true in general.

\begin{theorem}\label{sobo}
If $f\in W^{1,p}(\rn)$ has compact support, then
\begin{equation}\label{sobolim}
\lim_{s\to 1^-}(1-s)\,\int_{\rn}\int_{\rn} \frac{\abs{f(x)-f(y)}^p}{\norm{x-y}^{n+sp}_{K}} \,dx\,dy
=
\int_{\rn} \norm{\nabla f(x)}^p_{\ompd K} \,dx
\end{equation}
where $\,\ompd K$ is the polar $L_p$ moment body of $K$.
\end{theorem}

For the Euclidean $s$-seminorms and  the Euclidean unit ball $B$, the convex body $\ompd B$ is just a multiple of  $B$. Hence Theorem \ref{sobo}  recovers the result by Bourgain, Brezis \& Mironescu~\eqnref{bbm}  including the value of the constant $\alpha_{n,p}$.  For a convex body $K\subset \rn$, the polar $L_p$~moment body is the unit ball of the norm defined by
$$  \norm{v}^p_{\ompd K} =\frac{n+p}2\,\int_K \abs{v\cdot x}^p\,dx$$
for $v\in\rn$.

The polar body of $\omd K$, the convex body $\om K$, is the moment body of $K$. The convex body 
$$
\frac2{(n+1) \Vol{K}}\, \om K$$
is the centroid body of $K$, a classical concept that goes back at least to Dupin (cf.\ \cite{Gardner}). If we intersect the origin-symmetric convex body $K$ by halfspaces orthogonal to $u\in\sn$, then the centroids of these intersections  trace out the boundary of twice the centroid body of $K$, which explains the name centroid body. The name moment body comes from the fact that the corresponding moment vectors trace out the boundary (of a constant multiple) of $\om K$. Centroid bodies play an important role within the affine 
geometry of convex bodies (cf.\ \cite{Gardner, Lutwak90}) and moment bodies within the theory of valuations on convex bodies (see \cite{Ludwig:Minkowski, Ludwig:convex, Haberl_sln}). 

The polar body of $\ompd K$, the convex body $\omp K$, is the $L_p$ moment body of $K$ and 
$$
\frac 2{(n+p) \Vol{K}}\, \omp K$$
is the $L_p$ centroid body of $K$, a concept introduced by Lutwak \& Zhang \cite{LZ1997}. $L_p$ centroid bodies  and $L_p$ moment bodies have found important applications within convex geometry, probability theory, and the local theory of Banach spaces (cf.\ \cite{Haberl:Schuster1,LYZ2010a, LYZ04, LYZ2000, Paouris:Werner, Paouris06, FleuryGuedonPaouris, Ludwig:Minkowski, Parapatits:co, Parapatits:contravariant, LYZ2002, LYZ2000b, Ludwig:matrix}).

For $p>1$, it follows from Bourgain, Brezis \& Mironescu \cite[Theorem 2]{BourgainBrezisMironescu} that \eqnref{sobolim} also holds for $f\in L^p(\Omega)$ in the sense that if $f\not \in W^{1,p}(\Omega)$, then both sides of \eqnref{sobolim} are infinite.  For $p=1$, it follows from \cite[Theorem 3']{BourgainBrezisMironescu} that a corresponding result holds for $f\not \in BV(\rn)$ 
(see also D\'avila~\cite{Davila}). In \cite{Ludwig:fracperi}, the limiting behavior of fractional anisotropic Sobolev seminorms on $BV(\rn)$ is discussed using fractional anisotropic perimeters. 
Ponce \cite{Ponce2004} obtained several extensions of the results in \cite{BourgainBrezisMironescu}, from which Theorem~\ref{sobo} can also be deduced if anisotropic $s$-seminorms are used. The proof given in this paper is independent of Ponce's results. It makes use of the one-dimensional case of the Bourgain, Brezis \& Mironescu Theorem \eqnref{bbm} and the Blaschke-Petkanschin Formula from integral geometry.

Corresponding to the result of Maz$'$ya \& Shaposhnikova \eqnref{mazyash}, we obtain the following result.

\begin{theorem}\label{sobo0}
If $f\in W^{s,p}(\rn)$ for all $s\in(0,1)$ and $f$ has compact support, then
$$\lim_{s\to 0^+}\,\int_{\rn}\int_{\rn} \frac{\abs{f(x)-f(y)}^p}{\norm{x-y}^{n+sp}_{K}} \,dx\,dy
= \frac{2\, n}p\,
\Vol{K}\, \abs{f}^p_p.$$
\end{theorem}

\noindent
The proof of Theorem \ref{sobo0} is  based on the one-dimensional case of \eqnref{mazyash} and the Blaschke-Petkantschin Formula.

\section{Preliminaries}\label{tools}

We state the Blaschke-Petkantschin Formula (cf.~\cite[Theorem 7.2.7]{SchneiderWeil}) in the case in which it will be used. Let $H^k$ denote the $k$-dimensional Hausdorff measure on $\rn$ and let $\affn$ denote the affine Grassmannian of lines in $\rn$. 
 If $g:\rn\times\rn\to [0,\infty)$ is Lebesgue  measurable,
then
\begin{equation}\label{BP}
\int_\rn\int_\rn g(x,y)\,dH^n(x)\,dH^n(y) = \int_{\affn} \int_L \int_L g(x,y)\,\abs{x-y}^{n-1}\,dH^1(x)\,dH^1(y)\,dL,
\end{equation}
where $dL$ denotes integration with respect to a suitably normalized rigid motion invariant Haar measure on $\affn$. This measure can be described in the following way. Any line $L\in\affn$ can be parameterized using one of its direction unit vectors $u\in\sn$ and its base point $x\in u^\bot$, where $u^\bot$ is the hyperplane orthogonal to $u$, as $L=\{x+\lambda \, u: \lambda\in\R\}$. Hence, for $h:\affn\to[0,\infty)$ measurable,
\begin{equation}\label{ag}
\int_{\affn} h(L)\,dL= \frac 12 \int_\sn \int_{u^\bot} h(x+L_u)\,dH^{n-1}(x)\,dH^{n-1}(u),
\end{equation}
where $L_u=\{\lambda u:\lambda\in\R\}$.

\goodbreak

For $f\in W^{1,p}(\rn)$, we denote by $\bar f$  its precise representative (cf.\ \cite[Section 1.7.1]{Evans:Gariepy}). We require the following result. For every $u\in\sn$, the precise representative $\bar f$  is absolutely continuous on the lines $L=\{x+\lambda \, u: \lambda\in\R\}$ for $H^{n-1}$- a.e.\ $x\in u^\bot$  and 
its first-order (classical) partial derivatives belong to $L^p(\rn)$ (cf.\ \cite[Section 4.9.2, Theorem~2]{Evans:Gariepy}). Hence we have for the restriction
of $\bar f$ to $L$,
\begin{equation}\label{leoni}
\bar f|_L\in W^{1,p}(L)
\end{equation}
for a.e.\ line $L$ parallel to $u$. 

\goodbreak
\medskip

We require the following one-dimensional case of \eqnref{bbm}.

\begin{proposition}[\!\!\cite{BourgainBrezisMironescu}]\label{limone}
If $g\in W^{1,p}(\R)$ has compact support, then
$$\lim_{s\to 1^-} (1-s)  \int_{-\infty}^\infty \int_{-\infty}^\infty \frac{\abs{g(x)-g(y)}^p}{\abs{x-y}^{1+ps}}\,dx\,dy = \frac2 p\, \norm{g}^p_{W^{1,p}(\R)}.$$
\end{proposition}

We require the following one-dimensional case of \eqnref{mazyash}. 

\begin{proposition}[\!\!\cite{MazyaShaposhnikova}]\label{lim0one}
If $g\in W^{s,p}(\R)$ for all $s\in(0,1)$, then
$$\lim_{s\to 0^+} s  \int_{-\infty}^\infty \int_{-\infty}^\infty \frac{\abs{g(x)-g(y)}^p}{\abs{x-y}^{1+ps}}\,dx\,dy = \frac{4}{p} \, \abs{g}^p_{p}.$$
\end{proposition}

We also need the following result. The proof is based on the one-dimensional case of some estimates from \cite{BourgainBrezisMironescu}. Let $\diam(C) =\sup\{\abs{x-y}: x\in C, y\in C\}$ denote the diameter of $C\subset \R$.

\begin{lemma}\label{ab}
If $g\in W^{1,p}(\R)$ has compact support $C$, then there exists  a constant $\gamma_p$ depending only on $p$ such that
$$ (1-s) \int_{-\infty}^\infty \frac{\abs{g(x)-g(y)}^p}{\abs{x-y}^{1+ps}}\,dx\,dy \le \gamma_p\,\max(1,\diam(C))^{p}\, \norm{g}^p_{W^{1,p}(\R)}$$
for all $\,1/2\le s<1$.
\end{lemma}

\begin{proof}
If $g\in W^{1,p}(\R)$ is smooth, then  for $h\in \R$
$$g(x+h)-g(x)= h \int_0^1 g'(x+th)\,dt.$$
Hence  for $h\in\R$,
\begin{equation}\label{diff}
\int_{-\infty}^\infty \abs{g(x+h)-g(x)}^p \,dx \le \abs{h}^p \,\norm{g}^p_{W^{1,p}(\R)}.
\end{equation}
The same estimate is obtained for $g\in W^{1,p}(\R)$ by approximation
 (cf.\ \cite[Proposition 9.3]{Brezis}).
Let the support of $g$ be contained in $[-r,r]$, where $r\ge 1$.
By \eqnref{diff} we get
\begin{eqnarray*}
\int_{-\infty}^\infty\int_{-\infty}^\infty \frac{\abs{g(x)-g(y)}^p}{\abs{x-y}^{1+ps}}\,dx\,dy &=& 
\int_{-2r}^{-2r}\int_{-\infty}^\infty \frac{\abs{g(x+h)-g(x)}^p}{\abs{h}^{1+ps}}\,dx\,dh \\
&\le& \int_{-2r}^{2r} \abs{h}^{-(1-p(1-s))}dh\, \norm{g}^p_{W^{1,p}(\R)}\\
&\le& \frac{2\, (2r)^{p(1-s)}}{p(1-s)}\, \norm{g}^p_{W^{1,p}(\R)}.
\end{eqnarray*}
This completes the proof of the lemma.
\end{proof}

\goodbreak
The following  estimate is used in the proof of Theorem \ref{sobo0}.

\begin{lemma}\label{ab0}
If $g\in W^{s,p}(\R)$ for all $s\in(0,1)$, then 
$$ \int_{-\infty}^\infty\int_{-\infty}^\infty \frac{\abs{g(x)-g(y)}^p}{\abs{x-y}^{1+ps}}\,dx\,dy \le \frac{2^{p+1}}{ps}\, \abs{g}_p^p +  \norm{g}^p_{W^{s',p}(\R)}$$
for all $\,0< s\le s'<1$.
\end{lemma}

\begin{proof}
Note that
$$ \int_\R\int_{\{\abs{x-y}\ge 1\}} \frac{\abs{g(x)}^p}{\abs{x-y}^{1+ps}} \,dx\,dy
\le  \int_{\{\abs{z }\ge 1\}} \frac{dz}{\abs{z}^{1+ps}}\, \abs{g}_p^p  = 
\frac{2}{ps}\, \abs{g}_p^p.
$$
Hence, by Jensen's inequality,
\begin{equation*}
 \int\int_{\{\abs{x-y}\ge 1\}} \frac{\abs{g(x)-g(y)}^p}{\abs{x-y}^{1+ps}}\,dx\,dy \le
2^{p-1} \int_{-\infty}^\infty\int_{-\infty}^\infty \frac{\abs{g(x)}^p+ \abs{g(y)}^p}{\abs{x-y}^{1+ps}}\,dx\,dy
\le \frac{2^{p+1}}{ps}\, \abs{g}_p^p.
\end{equation*}
On the other hand, 
\begin{eqnarray*}
\int\!\!\int_{\{\abs{x-y}< 1\}} \frac{\abs{g(x)-g(y)}^p}{\abs{x-y}^{1+sp}} \,dx\,dy
&\le&  \int\!\!\int_{\{\abs{x-y}< 1\}} \frac{\abs{g(x)-g(y)}^p}{\abs{x-y}^{1+s'p}} \,dx\,dy
\end{eqnarray*}
for $0<s<s'$.
\end{proof}

\section{Proof of Theorem \ref{sobo}}
By the Blaschke-Petkantschin Formula \eqnref{BP}, we have
\begin{equation}\label{first}
\int\limits_\rn\int\limits_{\rn} \frac{\abs{f(x)-f(y)}^p}{\norm{x-y}_K^{n+ps}} \,dx\,dy
 =\!\!\! \int\limits_{\affn} \norm{u(L)}_K^{-(n+ps)}\!\int\limits_{L} \int\limits_{L} \frac{\abs{f(x)-f(y)}^p}{\abs{x-y}^{1+ps}} \,dH^1(x)\,dH^1(y)\,dL.
\end{equation}
By Proposition \ref{limone} and \eqnref{leoni}, we have
\begin{equation}\label{one}
\lim_{s\to1^-} (1-s)
\int\limits_L\int\limits_{L} \frac{\abs{f(x)-f(y)}^p}{\abs{x-y}^{1+ps}} dx\,dy =
\frac2 p\, \int\limits_{L} \abs{\nabla f(x)\cdot u}^p\, dH^1(x)
\end{equation}
for a.e.\ line $L$ parallel to $u\in\sn$.
\goodbreak

By Fubini's Theorem, the definition of the measure on the affine Grassmannian \eqnref{ag} and the polar coordinate formula, we get 
\begin{align*}
\frac 2p \!\!\int\limits_{\affn} & \norm{u(L)}_K^{-(n+p)}\int\limits_{L} 
\abs{\nabla f(x)\cdot u}^p\, dH^1(x)\,dL\\
= &\,\, \frac1p
\int\limits_{\sn} \int\limits_{ u^\bot}  \norm{u}_K^{-(n+p)} \int\limits_{y+L_u} 
\abs{\nabla f(x)\cdot u}^p\, dH^1(x)\,dH^{n-1}(y)\,dH^{n-1}(u)\\ =
&\,\, \frac 1p
\int\limits_{\sn} \int\limits_{\rn} \norm{u}_K^{-(n+p)}\,
\abs{\nabla f(x)\cdot u}^p\,dH^{n}(x)\,dH^{n-1}(u)\\ =
&\,\, \frac{n+p}p
\int\limits_{K} \int\limits_{\rn} \,
\abs{\nabla f(x)\cdot y}^p\,dH^{n}(x)\,dH^{n}(y).
\end{align*}
Using Fubini's Theorem and  the definition of the $L_p$ moment body of $K$, we obtain
\begin{align}\label{second}
\int\limits_{\affn} & \norm{u(L)}_K^{-(n+ps)}\int\limits_{L} 
\abs{\nabla f(x)\cdot u}^p\, dH^1(x)\,dL =
\int\limits_\rn \norm{ \nabla f(x)}_{\ompd K}^p \,dx.
\end{align}
So, in particular, we have
\begin{equation}\label{side2}
\int\limits_{\affn} \int\limits_{L} 
\abs{\nabla f(x)\cdot u}^p\, dH^1(x)\,dL=
\alpha_{n,p} \,\int\limits_\rn \abs{ \nabla f(x)}^p \,dx<\infty,
\end{equation}
where $\alpha_{n,p}$ is a constant. 

Using the Dominated Convergence Theorem combined with Lemma \ref{ab} and \eqnref{side2}, we obtain 
from \eqnref{first}, \eqnref{one}  and \eqnref{second} that
\begin{align*}
\lim_{s\to1^-} (1-s)
\int\limits_\rn\int\limits_{\rn} \frac{\abs{f(x)-f(y)}^p}{\norm{x-y}_K^{n+ps}}\, dx\,dy = 
\int\limits_\rn \norm{ \nabla f(x)}_{\ompd K}^p \,dx.
\end{align*}
This concludes the proof of the theorem.

\section{Proof of Theorem \ref{sobo0}}
By the Blaschke-Petkantschin Formula  \eqnref{BP}, we obtain
\begin{equation*}
\int\limits_\rn\int\limits_{\rn} \frac{\abs{f(x)-f(y)}^p}{\norm{x-y}_K^{n+ps}}\, dx\,dy = 
\int\limits_{\affn} \norm{u(L)}_K^{-(n+ps)}
 \int\limits_{L} \int\limits_{L} \frac{\abs{f(x)-f(y)}^p}{\abs{x-y}^{s+1}} \,dx\,dy\,dL.
\end{equation*}
Thus we obtain by the Dominated Convergence Theorem, Lemma \ref{ab0} and Proposition \ref{lim0one} that
$$\lim_{s\to0^+} s\,
\int\limits_\rn\int\limits_{\rn} \frac{\abs{f(x)-f(y)}^p}{\norm{x-y}_K^{n+ps}} dx\,dy =\frac4p
\int\limits_{\affn} \norm{u(L)}_K^{-n}  \int\limits_{L} 
\abs{f(x)}^p\, dH^1(x)\,dL.$$
By Fubini's Theorem, the definition of the measure on the affine Grassmannian \eqnref{ag} and the polar coordinate formula for volume, we get 
\begin{align*}
\frac4p \int\limits_{\affn} & \norm{u(L)}_K^{-n}  \int\limits_{L} 
\abs{f(x)}^p\, dH^1(x)\,dL\\= &\,\, \frac2p
\int\limits_{\sn} \int\limits_{ u^\bot}  \norm{u}_K^{-n} \int\limits_{y+L_u} 
\abs{f(x)}^p\, dH^1(x)\,dH^{n-1}(y)\,dH^{n-1}(u)\\ =
&\,\, \frac2p
\int\limits_{\sn} \int\limits_{\rn}  \norm{u}_K^{-n} \,
\abs{ f(x)}^p\,dH^{n}(x)\,dH^{n-1}(u)\\ =
&\,\, \frac{2 n}p\,
\Vol{K}\,\, \abs{f}_p^p.
\end{align*}
This concludes the proof of the theorem.

\bigskip
\parindent=0pt
\begin{samepage}
Institut f\"ur Diskrete Mathematik und Geometrie\\
Technische Universit\"at Wien\\
Wiedner Hauptstr.\ 8-10/1046\\
1040 Wien, Austria\\
E-mail: monika.ludwig@tuwien.ac.at
\end{samepage}


\footnotesize
\begin{thebibliography}{10}

\bibitem{AlvinoFeroneTrombettiLions}
A.~Alvino, V.~Ferone, G.~Trombetti, and P.-L. Lions, {\em Convex symmetrization
  and applications}, Ann. Inst. H. Poincar\'e Anal. Non Lin\'eaire {\bf 14}
  (1997),  275--293.

\bibitem{BourgainBrezisMironescu}
J.~Bourgain, H.~Brezis, and P.~Mironescu, {\em Another look at {S}obolev
  spaces}, In: {O}ptimal Control and Partial Differential Equations ({J}. {L}.
  {M}enaldi, {E}. {R}ofman and {A}. {S}ulem, eds.). {A} volume in honor of {A}.
  {B}ensoussans's 60th birthday, {Amsterdam: IOS Press; Tokyo: Ohmsha}, 2001.

\bibitem{BourgainBrezisMironescu02}
J.~Bourgain, H.~Brezis, and P.~Mironescu, {\em Limiting embedding theorems for
  {$W\sp {s,p}$} when {$s\uparrow1$} and applications}, J. Anal. Math. {\bf 87}
  (2002), 77--101, Dedicated to the memory of Thomas H. Wolff.

\bibitem{Brezis}
H.~Brezis, {\em Functional {A}nalysis, {S}obolev {S}paces and {P}artial
  {D}ifferential {E}quations}, Universitext, Springer, New York, 2011.

\bibitem{Cordero:Nazaret:Villani}
D.~Cordero-Erausquin, B.~Nazaret, and C.~Villani, {\em A mass-transportation
  approach to sharp {S}obolev and {G}agliardo-{N}irenberg inequalities}, Adv.
  Math. {\bf 182} (2004), 307--332.

\bibitem{Davila}
J.~D{\'a}vila, {\em On an open question about functions of bounded variation},
  Calc. Var. Partial Differential Equations {\bf 15} (2002), 519--527.

\bibitem{hitchhiker}
E.~Di~Nezza, G.~Palatucci, and E.~Valdinoci, {\em Hitchhiker's guide to the
  fractional {S}obolev spaces}, Bull. Sci. Math. {\bf 136} (2012), 521--573.

\bibitem{Evans:Gariepy}
L.~Evans and R.~Gariepy, {\em Measure theory and fine properties of functions},
  Studies in Advanced Mathematics, CRC Press, Boca Raton, FL, 1992.

\bibitem{FigalliMaggiPratelli}
A.~Figalli, F.~Maggi, and A.~Pratelli, {\em Sharp stability theorems for the
  anisotropic {S}obolev and log-{S}obolev inequalities on functions of bounded
  variation}, Adv. Math., in press.

\bibitem{FleuryGuedonPaouris}
B.~Fleury, O.~Gu{\'e}don, and G.~Paouris, {\em A stability result for mean
  width of {$L\sb p$}-centroid bodies}, Adv. Math. {\bf 214} (2007), 865--877.

\bibitem{Gagliardo}
E.~Gagliardo, {\em Caratterizzazioni delle tracce sulla frontiera relative ad
  alcune classi di funzioni in {$n$} variabili}, Rend. Sem. Mat. Univ. Padova
  {\bf 27} (1957), 284--305.

\bibitem{Gardner}
R.~Gardner, {\em Geometric {T}omography}, second ed., Encyclopedia of
  Mathematics and its Applications, vol.~58, Cambridge University Press,
  Cambridge, 2006.

\bibitem{Gromov_Sobolev}
M.~Gromov, {\em Isoperimetric inequalities in {R}iemannian manifolds},
  {Appendix of: Asymptotic Theory of Finite-dimensional Normed Spaces (V. D. Milman
	\& G. Schechtman)},
  Springer-Verlag, 1986.

\bibitem{Haberl_sln}
C.~Haberl, {\em Minkowski valuations intertwining the special linear group}, J.
  Eur. Math. Soc. (JEMS) (2012), 1565--1597.

\bibitem{Haberl:Schuster1}
C.~Haberl and F.~Schuster, {\em General ${L}_p$ affine isoperimetric
  inequalities}, J. Differential Geom. {\bf 83} (2009), 1--26.

\bibitem{Ludwig:matrix}
M.~Ludwig, {\em Ellipsoids and matrix valued valuations}, Duke Math. J. {\bf
  119} (2003), 159--188.

\bibitem{Ludwig:Minkowski}
M.~Ludwig, {\em Minkowski valuations}, Trans. Amer. Math. Soc. {\bf 357}
  (2005), 4191--4213.

\bibitem{Ludwig:convex}
M.~Ludwig, {\em Minkowski areas and valuations}, J. Differential Geom. {\bf 86}
  (2010), 133--161.

\bibitem{Ludwig:fracperi}
M.~Ludwig, {\em Anisotropic fractional perimeters}, preprint.

\bibitem{Lutwak90}
E.~Lutwak, {\em Centroid bodies and dual mixed volumes}, Proc. London Math.
  Soc. {\bf 60} (1990), 365--391.


\bibitem{LYZ2000}
E.~Lutwak, D.~Yang, and G.~Zhang, {\em ${L}\sb p$ affine isoperimetric
  inequalities}, J. Differential Geom. {\bf 56} (2000), 111--132.

\bibitem{LYZ2000b}
E.~Lutwak, D.~Yang, and G.~Zhang, {\em A new ellipsoid associated with convex
  bodies}, Duke Math. J. {\bf 104} (2000), 375--390.

\bibitem{LYZ2002}
E.~Lutwak, D.~Yang, and G.~Zhang, {\em The {C}ramer-{R}ao inequality for star
  bodies}, Duke Math. J. {\bf 112} (2002), 59--81.

\bibitem{LYZ04}
E.~Lutwak, D.~Yang, and G.~Zhang, {\em Moment-entropy inequalities}, Ann.
  Probab. {\bf 32} (2004), 757--774.

\bibitem{LYZ2010a}
E.~Lutwak, D.~Yang, and G.~Zhang, {\em Orlicz centroid bodies}, J. Differential
  Geom. {\bf 84} (2010), 365--387.

\bibitem{LZ1997}
E.~Lutwak and G.~Zhang, {\em Blaschke-{S}antal\'o inequalities}, J.
  Differential Geom. {\bf 47} (1997), 1--16.

\bibitem{Mazya}
V.~Maz'ya, {\em Sobolev {S}paces with {A}pplications to {E}lliptic {P}artial
  {D}ifferential {E}quations}, augmented ed., Grundlehren der Mathematischen
  Wissenschaften, vol. 342, Springer, Heidelberg, 2011.

\bibitem{MazyaShaposhnikova}
V.~Maz$'$ya and T.~Shaposhnikova, {\em On the {B}ourgain, {B}rezis, and
  {M}ironescu theorem concerning limiting embeddings of fractional {S}obolev
  spaces}, J. Funct. Anal. {\bf 195} (2002), 230--238.

\bibitem{Paouris06}
G.~Paouris, {\em Concentration of mass on convex bodies}, Geom. Funct. Anal.
  {\bf 16} (2006),  1021--1049.

\bibitem{Paouris:Werner}
G.~Paouris and E.~Werner, {\em Relative entropy of cone measures and {$L\sb p$}
  centroid bodies}, Proc. Lond. Math. Soc. (3) {\bf 104} (2012), 253--286.

\bibitem{Parapatits:contravariant}
L.~Parapatits, {\em {S}{L}$(n)$-contravariant ${L}_p$-{M}inkowski valuations},
  Trans. Amer. Math. Soc., in press.

\bibitem{Parapatits:co}
L.~Parapatits, {\em {S}{L}$(n)$-covariant ${L}_p$-{M}inkowski valuations},
  preprint.

\bibitem{Ponce2004}
A.~Ponce, {\em A new approach to {S}obolev spaces and connections to
  {$\Gamma$}-convergence}, Calc. Var. Partial Differential Equations {\bf 19}
  (2004), 229--255.

\bibitem{SchneiderWeil}
R.~Schneider and W.~Weil, {\em Stochastic and {I}ntegral {G}eometry},
  Probability and its Applications (New York), Springer-Verlag, Berlin, 2008.

\end{thebibliography}


\normalsize
\end{document}